\documentclass[11pt]{article}

\usepackage[T1]{fontenc}
\usepackage[utf8]{inputenc}
\usepackage{lmodern}
\usepackage{amsmath,amssymb,amsthm,mathtools}
\usepackage{geometry}
\usepackage{enumitem}
\usepackage{microtype}
\usepackage{hyperref}
\usepackage{cleveref}
\usepackage{aliascnt}
\usepackage{mathrsfs}

\hypersetup{colorlinks=true,linkcolor=blue,citecolor=blue,urlcolor=blue}

\newtheorem{theorem}{Theorem}[section]
\newaliascnt{proposition}{theorem}
\newtheorem{proposition}[proposition]{Proposition}
\aliascntresetthe{proposition}
\newaliascnt{lemma}{theorem}
\newtheorem{lemma}[lemma]{Lemma}
\aliascntresetthe{lemma}
\newaliascnt{corollary}{theorem}
\newtheorem{corollary}[corollary]{Corollary}
\aliascntresetthe{corollary}
\newaliascnt{remark}{theorem}
\newtheorem{remark}[remark]{Remark}
\aliascntresetthe{remark}
\theoremstyle{definition}
\newaliascnt{definition}{theorem}
\newtheorem{definition}[definition]{Definition}
\aliascntresetthe{definition}
\newaliascnt{question}{theorem}

\aliascntresetthe{question}

\crefname{theorem}{theorem}{theorems}
\Crefname{theorem}{Theorem}{Theorems}
\crefname{proposition}{proposition}{propositions}
\Crefname{proposition}{Proposition}{Propositions}
\crefname{lemma}{lemma}{lemmas}
\Crefname{lemma}{Lemma}{Lemmas}
\crefname{corollary}{corollary}{corollaries}
\Crefname{corollary}{Corollary}{Corollaries}
\crefname{remark}{remark}{remarks}
\Crefname{remark}{Remark}{Remarks}
\crefname{definition}{definition}{definitions}
\Crefname{definition}{Definition}{Definitions}
\crefname{question}{Question}{Questions}
\Crefname{question}{Question}{Questions}

\newcommand{\Met}{\mathscr{M}^{\infty}(M)}
\newcommand{\NoFoc}{\mathscr{F}(M)}
\newcommand{\NoConj}{\mathscr{C}(M)}

\newcommand{\Z}{\mathcal{Z}}
\newcommand{\eps}{\varepsilon}

\title{Isolated Flat Points and $C^2$-Robustness\\of the No Focal Points Property}
\author{Alexander Cantoral \and Sergio Roma\~na}
\date{}

\begin{document}
\maketitle
\begin{center}
    \emph{Dedicated to Nestor Nina Zárate, whose question inspired this work.}
\end{center}
\begin{abstract}
We prove a local stability criterion for the no focal points property on compact Riemannian surfaces. More precisely, if a smooth metric $g$ has non-positive Gaussian curvature and its zero-curvature set is finite, then $g$ belongs to the $C^2$-interior of the set of metrics without focal points. Thus, every sufficiently small $C^2$-perturbation of $g$ still has no focal points, although arbitrarily small perturbations may create regions of positive Gaussian curvature. By Ruggiero's characterization of the $C^2$-interior of the set of metrics without conjugate points, all metrics in the resulting neighborhood are Anosov. We also show that, starting from any hyperbolic metric on a compact surface, one can prescribe an arbitrary finite set as the zero set of the Gaussian curvature of a smooth conformal non-positively curved metric. These metrics can be approximated smoothly by negatively curved metrics, so they lie on the boundary of the negatively curved regime while remaining interior points of the no-focal-points regime. The proof of the stability theorem combines uniform local convexity, Gulliver's bound on the length of geodesic segments contained in small balls, and an inductive argument on the Riccati equation  that controls successive passages through the regions where positive curvature may appear.
\end{abstract}



\section{Introduction}

Anosov metrics, namely Riemannian metrics whose geodesic flows are Anosov, form one of the basic classes connecting hyperbolic dynamics and Riemannian geometry. On a closed manifold, strictly negative sectional curvature implies that the geodesic flow is Anosov \cite{Anosov1967}. The converse is false: the class of Anosov metrics is strictly larger than the class of negatively curved metrics; see, for instance, \cite{Eberlein1973,Gulliver1975}. A classical theorem of Klingenberg shows that an Anosov geodesic flow has no conjugate points \cite{Klingenberg1974}. In the opposite direction, Ruggiero proved a remarkable stability characterization: on a compact manifold, a metric which is a $C^2$-interior point of the set of metrics without conjugate points has Anosov geodesic flow \cite{Ruggiero1991}. Together with the $C^2$-openness of the Anosov property, this identifies the $C^2$-interior  of the set of metrics without conjugate points with the set of Anosov metrics.

The no focal points condition lies naturally between curvature restrictions and the absence of conjugate points. If $M$ is a Riemannian manifold and $J$ is a non-trivial normal Jacobi field along a unit-speed geodesic $\gamma$ with $J(0)=0$, the no focal points condition requires
\[
    \frac{d}{dt}\|J(t)\|^2>0,\qquad t>0.
\]
In particular, if $K$ denotes the sectional curvature of $M$,
\[
    K\le 0
    \quad\Longrightarrow\quad
    \text{no focal points}
    \quad\Longrightarrow\quad
    \text{no conjugate points}.
\]
Neither converse holds in general. Gulliver constructed compact manifolds whose sectional curvature takes both signs and which nevertheless have no focal points; he also constructed Anosov metrics with focal points \cite{Gulliver1975}. Moreover, the final stability observation in \cite{Gulliver1975} points out that the strict inequalities used in his constructions persist under sufficiently small $C^2$-perturbations. Thus, his examples already produce $C^2$-open families of metrics without focal points that contain positive curvature.



The purpose of the present paper is therefore not merely to produce another mixed-curvature open set. Instead, we study a finer local phenomenon that is not visible from the blister construction: can a metric which is itself non-positively curved, and which lies at the boundary of the strictly negative curvature regime, nevertheless be an interior point of the set of metrics without focal points?

Our main theorem gives a positive answer whenever the flat locus consists of finitely many isolated points. Let $\Met$ denote the space of smooth Riemannian metrics on a fixed compact surface $M$, endowed with the $C^2$-topology, and let $\NoFoc\subset\Met$ denote the subset of metrics without focal points.

\begin{theorem}[Robustness at a finite flat locus]\label{thm:main}
Let $(M,g_*)$ be a compact Riemannian surface such that
\[
    K_{g_*}\le 0
\]
and assume that the zero-curvature set
\[
    \Z(g_*):=\{x\in M:K_{g_*}(x)=0\}
\]
is finite. Then $g_*$ belongs to the $C^2$-interior of $\NoFoc$. Equivalently, there exists a $C^2$-neighborhood $\mathcal U$ of $g_*$ in $\Met$ such that every $g\in\mathcal U$ has no focal points.
\end{theorem}

This theorem is local in the space of metrics and relies only on the geometry of $g_*$ near its finite flat locus. The perturbations in $\mathcal U$ are arbitrary; in particular, they are not assumed to remain non-positively curved. Since the curvature of $g_*$ vanishes at each point of $\Z(g_*)$, a small conformal perturbation can make the curvature positive near any chosen flat point. Consequently, every sufficiently small neighborhood in \cref{thm:main} contains metrics whose curvature changes sign; no restriction to non-positive-curvature perturbations is required.

\begin{corollary}[Mixed curvature in the robust neighborhood]\label{cor:mixed}
Under the assumptions of \cref{thm:main}, if $\Z(g_*)\neq\varnothing$, then every $C^2$-neighborhood of $g_*$ contains a smooth metric whose Gaussian curvature takes both positive and negative values. In particular, the neighborhood $\mathcal U$ in \cref{thm:main} may be chosen so that it contains metrics of mixed curvature, all of them without focal points.
\end{corollary}

The dynamical consequence is immediate but important. Since no focal points implies no conjugate points, every metric in the open set $\mathcal U$ is itself a $C^2$-interior point of the set of metrics without conjugate points. Ruggiero's theorem therefore yields the following corollary.

\begin{corollary}[Anosov consequence]\label{cor:anosov}
Every metric in the neighborhood $\mathcal U$ of \cref{thm:main} has Anosov geodesic flow.
\end{corollary}

Thus, \cref{thm:main} produces an open set lying simultaneously inside the no-focal-points and Anosov regions while crossing the boundary of the pointwise curvature condition $K<0$.

Our second main result shows that the geometric hypothesis of \cref{thm:main} occurs abundantly and can be realized with a prescribed finite flat set. The construction is conformal and starts from any hyperbolic metric.

\begin{theorem}[Prescribed isolated flat points]\label{thm:construction}
Let $(M,g_0)$ be a compact hyperbolic surface and let
\[
    Z=\{p_1,\dots,p_N\}\subset M
\]
be a non-empty finite set. Then there exists a smooth function $u\colon M\to\mathbb R$ such that the conformal metric
\[
    g_Z=e^{2u}g_0
\]
has non-positive Gaussian curvature and
\[
    K_{g_Z}(x)=0
    \quad\Longleftrightarrow\quad
    x\in Z.
\]
Moreover, $g_Z$ can be approximated in the $C^\infty$-topology by negatively curved conformal metrics.
\end{theorem}

Combining \cref{thm:main,thm:construction}, we conclude that every compact hyperbolic surface admits non-positively curved metrics with any prescribed finite set of isolated flat points whose absence of focal points is robust in the $C^2$-topology. Since the metrics of \cref{thm:construction} are smooth limits of negatively curved metrics but are not themselves negatively curved, they lie on the boundary of the negatively curved region. At the same time, \cref{thm:main} places them in the interior of the no-focal-points region. This separation between a pointwise curvature condition and a Jacobi-field condition is the central phenomenon of the paper.

The proof of \cref{thm:main} follows a robust Riccati mechanism. Around each flat point we choose a small ball in which a nearby metric may have a small amount of positive curvature. Gulliver's length estimate gives a uniform upper bound on the duration of each geodesic passage through such a ball. Uniform strict convexity of a larger concentric ball prevents a geodesic from leaving and immediately returning to the same perturbed region, while the finiteness and separation of the flat points gives a uniform lower bound between passages through different regions. Outside the union of the small balls the curvature of every nearby metric remains uniformly negative. For a unit-speed geodesic $\gamma$, let $J$ be a normal Jacobi field along $\gamma$ with $J(0)=0$. Write $J(t)=y(t)E(t)$, where $E$ is a parallel unit normal field along $\gamma$. We then study
\[
    v(t)=\frac{y'(t)}{y(t)},
    \qquad
    v'(t)=-K(\gamma(t))-v(t)^2.
\]
The solution may decrease while the geodesic crosses a region of small positive curvature, but the crossing time is uniformly short. Between crossings, uniform negative curvature forces the Riccati solution to recover a fixed positive level. An induction over all visits keeps $v(t)$ positive for all $t>0$, which is precisely the no focal points condition.

The finiteness assumption in \cref{thm:main} has geometric significance. If a non-positively curved surface contains a closed geodesic along which the curvature vanishes identically, the transverse Jacobi equation along that geodesic has a neutral solution. Hence the geodesic flow is not Anosov. By Ruggiero's theorem, such a metric cannot be a $C^2$-interior point of the set of metrics without conjugate points, and therefore cannot be $C^2$-robustly without focal points. Thus one cannot replace the isolated flat locus in our construction by a flat closed geodesic while retaining the same robustness conclusion.

The paper is organized as follows. In \cref{sec:prelim}, we recall the Jacobi and Riccati formulations of the no focal points property and establish the uniform local geometry used later. In \cref{sec:construction}, we prove \cref{thm:construction}. In \cref{sec:robustness}, we prove the stability theorem (\cref{thm:main}) by the Riccati induction described above. Finally, in \cref{sec:consequences}, we prove \cref{cor:mixed,cor:anosov} and present several geometric consequences of the construction.

\section{Preliminaries and uniform local geometry}\label{sec:prelim}

Throughout the paper, $M$ denotes a compact smooth surface without boundary. All metrics are smooth, and neighborhoods in the space
of metrics are taken in the $C^2$-topology. For a metric $g$, we denote by $K_g$ its Gaussian curvature, by $d_g$ its distance function, and by $B_r^g(p)$ the open metric ball of radius $r$ centered at $p$.

\subsection{Jacobi fields, focal points, and the Riccati equation}

Let $\gamma\colon\mathbb R\to M$ be a unit-speed geodesic. A vector field $J$ along $\gamma$ is a Jacobi field if
\[
    J''(t)+R(\gamma'(t),J(t))\gamma'(t)=0.
\]
On a surface, every normal Jacobi field can be written as
\[
    J(t)=y(t)E(t),
\]
where $E$ is a parallel unit normal vector field along $\gamma$. The Jacobi equation then reduces to the scalar equation
\begin{equation*}
    y''(t)+K_g(\gamma(t))y(t)=0.
\end{equation*}

\begin{definition}
A complete Riemannian manifold has \emph{no focal points} if, for every unit-speed geodesic $\gamma$ and every non-trivial Jacobi field $J$ along $\gamma$ satisfying $J(0)=0$, one has
\[
    \frac{d}{dt}\|J(t)\|^2>0
    \qquad\text{for all }t>0.
\]
\end{definition}

It is enough to verify this condition for normal Jacobi fields. Indeed, the normal and tangential components of a Jacobi field are orthogonal and the tangential component vanishing at $t=0$ is affine in $t$, so the normal component contains the only possible obstruction.

After rescaling, we may therefore consider the solution of
\begin{equation}\label{eq:ivp}
    y''(t)+K_g(\gamma(t))y(t)=0,
    \qquad y(0)=0,\quad y'(0)=1.
\end{equation}
For $t>0$ sufficiently small, $y(t)>0$. Let
\[
    T:=\sup\bigl\{s>0:y(t)>0\text{ for every }0<t<s\bigr\}\in(0,+\infty].
\]
On $(0,T)$, define the Riccati variable
\begin{equation}\label{eq:riccati-def}
    v(t):=\frac{y'(t)}{y(t)}.
\end{equation}
Then
\begin{equation}\label{eq:riccati}
    v'(t)=-K_g(\gamma(t))-v(t)^2.
\end{equation}
Moreover,
\begin{equation}\label{eq:norm-derivative}
    \frac{d}{dt}\|J(t)\|^2
    =2y(t)y'(t)
    =2v(t)y(t)^2.
\end{equation}
Thus, it suffices to prove that $T=+\infty$ and $v(t)>0$ for every $t>0$.

If $K_g\le0$, equation \eqref{eq:riccati} already reflects the convexity responsible for the no focal points property. In the perturbative situation considered here, $K_g$ may become positive near finitely many points, and the purpose of the later estimates is to ensure that the corresponding temporary decrease of $v$ cannot bring it to zero.

\subsection{Uniform normal balls and convexity}

For $p\in M$, let
\[
    f_{g,p}(x):=\frac12 d_g(p,x)^2.
\]
On a normal ball centered at $p$, this function is smooth. In $g$-normal coordinates centered at $p$,
\[
    \operatorname{Hess}_g f_{g,p}\big|_p=g_p.
\]
The following uniform version will be used repeatedly.

\begin{lemma}[Uniform local convexity]\label{lem:uniform-convexity}
Let $g_*$ be a smooth metric on $M$, and let $Z=\{p_1,\dots,p_N\}\subset M$ be finite. There exist $r_*>0$, $c>0$, and a $C^2$-neighborhood $\mathcal V$ of $g_*$ such that, for every $g\in\mathcal V$ and every $i=1,\ldots,N$, the following hold:
\begin{enumerate}[label=(\roman*)]
    \item $B_{r_*}^g(p_i)$ is a normal ball;
    \item the balls $B_{r_*}^g(p_i)$ are pairwise disjoint;
    \item on $B_{r_*}^g(p_i)$, one has
    \[
        \operatorname{Hess}_g f_{g,p_i}\ge c\,g
    \]
    as quadratic forms.
\end{enumerate}
\end{lemma}

\begin{proof}
For the fixed metric $g_*$, choose $r_*>0$ smaller than the injectivity radius at each $p_i$ and small enough that the corresponding closed balls are pairwise disjoint. Since
\[
    \operatorname{Hess}_{g_*}f_{g_*,p_i}\big|_{p_i}=(g_*)_{p_i},
\]
continuity implies, after decreasing $r_*$ if necessary, that
\[
    \operatorname{Hess}_{g_*}f_{g_*,p_i}\ge 2c\,g_*
\]
on $B_{r_*}^{g_*}(p_i)$, for some $c>0$ independent of $i$.

The geodesic equation, the exponential map, and its differential depend continuously on the metric in the $C^2$-topology on compact subsets. Since the set of centers is finite, after shrinking to a sufficiently small $C^2$-neighborhood $\mathcal V$, the same radius remains smaller than the injectivity radius at each $p_i$, the corresponding balls remain pairwise disjoint, and the Hessians of the squared distance functions remain uniformly positive definite. Shrinking $c$ if necessary yields the stated estimate with respect to $g$, uniformly for all $g\in\mathcal V$.
\end{proof}

The next estimate is due to Gulliver and is one of the key geometric inputs in our proof.

\begin{lemma}[Gulliver's passage-length estimate]\label{lem:gulliver}
Let $(M,g)$ be a Riemannian manifold with sectional curvature bounded above by $C>0$. Suppose $B_r^g(p)$ is a normal ball and
\[
    r\sqrt C<\frac{\pi}{2}.
\]
Then every geodesic segment entirely contained in $B_r^g(p)$ has length at most $2r$.
\end{lemma}

\begin{proof}
This is theorem 1 of Gulliver \cite{Gulliver1975}. The upper curvature bound and the radius condition guarantee the required radial convexity, and Gulliver proves that the longest geodesic segment contained in the ball has length equal to its diameter.
\end{proof}

We shall also use the elementary fact that $C^2$-close metrics on a compact manifold are uniformly equivalent. In particular, after shrinking a $C^2$-neighborhood of $g_*$, the inclusions 
\begin{equation}\label{eq:ball-inclusion}
    B_{a/2}^{g_*}(p)\subset B_a^g(p)\subset B_{2a}^{g_*}(p),
\end{equation}
hold for every $g\in\mathcal V$, every $a>0$, and each of the finitely many centers under consideration.

\section{Conformal metrics with prescribed isolated flat points}\label{sec:construction}

We now prove \cref{thm:construction}. The argument is elementary but useful, as it allows us to prescribe the flat locus exactly.

\begin{proof}[Proof of \cref{thm:construction}]
Let $Z=\{p_1,\dots,p_N\}\subset M$ be a finite set. Choose pairwise disjoint $g_0$-normal balls
\[
    B_{R_i}^{g_0}(p_i),\qquad i=1,\dots,N.
\]
For each $i$, let
\[
    r_i(x):=d_{g_0}(p_i,x).
\]
Choose a smooth cutoff function $\xi_i\colon[0,+\infty)\to[0,1]$ such that
\[
    \xi_i(t)=1\quad\text{for }0\le t\le R_i/3,
    \qquad
    \xi_i(t)=0\quad\text{for }t\ge 2R_i/3.
\]
Define $h_i\colon M\to \mathbb R$ as
\begin{equation*}    h_i(x):=\xi_i(r_i(x))\,r_i(x)^2+1-\xi_i(r_i(x)).
\end{equation*}
The function $h_i$ is smooth on $M$. Indeed, near $p_i$, it is the smooth squared distance $r_i^2$; on the transition annulus, the distance function is smooth; and outside the normal ball the cutoff function vanishes. Furthermore,
\[
    h_i\ge0,
    \qquad
    h_i(x)=0\Longleftrightarrow x=p_i.
\]
Set
\begin{equation*}
    h:=\prod_{i=1}^N h_i.
\end{equation*}
Then $h\in C^\infty(M)$ and
\begin{equation}\label{eq:h-zero-set}
    h\ge 0, \qquad h^{-1}(0)=Z.
\end{equation}

Let
\begin{equation*}
    \alpha:=\frac{1}{\operatorname{Area}_{g_0}(M)}
    \int_M h\,dV_{g_0}>0
\end{equation*}
be the average of $h$, and define the function $F:M\to \mathbb R$ as
\begin{equation}\label{eq:F}
    F(x):=\frac{h(x)}{\alpha}-1.
\end{equation}
By construction,
\begin{equation}\label{eq:integral}
\int_M F \,dV_{g_0}
= \frac{1}{\alpha}\int_M h\,dV_{g_0} - \int_M 1 \, dV_{g_0}=0.
\end{equation}
Since $M$ is a closed surface and $F\in C^{\infty}(M)$ satisfies \eqref{eq:integral}, the Poisson equation
\begin{equation}\label{eq:poisson}
    \Delta_{g_0}u=F
\end{equation}
admits a smooth solution $u$, which is unique under the normalization
\[
    \int_M u\,dV_{g_0}=0
\]
(see, for example, \cite{Aubin1998}). Define the metric
\[
    g_Z:=e^{2u}g_0.
\]
Since $u$ is smooth, $g_Z$ is a smooth Riemannian metric conformal to $g_0$. For the Laplacian convention used here, the Gaussian curvature under a conformal change satisfies
\begin{equation*}
    K_{g_Z}=e^{-2u}\bigl(K_{g_0}-\Delta_{g_0}u\bigr).
\end{equation*}
Since $g_0$ is hyperbolic, after scaling, we may assume $K_{g_0}\equiv-1$. Hence, using \eqref{eq:F} and \eqref{eq:poisson},
\begin{align*}
    K_{g_Z}
    &=e^{-2u}(-1-F)\\
    &=-\frac{e^{-2u}}{\alpha}\,h.
\end{align*}
Since $\alpha > 0$ and $h\ge 0$ on $M$, we obtain
\[
    K_{g_Z}\le0
\]
and, by \eqref{eq:h-zero-set},
\[
    K_{g_Z}(x)=0
    \quad\Longleftrightarrow\quad
    x\in Z.
\]
This proves the first part of the theorem.

It remains to show that $g_Z$ can be approximated by negatively curved metrics. For $s>0$, set
\[
    h_s:=h+s,
    \qquad
    \alpha_s:=\frac{1}{\operatorname{Area}_{g_0}(M)}\int_M h_s\,dV_{g_0}=\alpha+s,
\]
and define
\[
    F_s:=\frac{h_s}{\alpha_s}-1.
\]
Again $\int_MF_s\,dV_{g_0}=0$. Let $u_s$ be the normalized solution of
\[
    \Delta_{g_0}u_s=F_s,
    \qquad
    \int_Mu_s\,dV_{g_0}=0.
\]
Define the metric $g_s:=e^{2u_s}g_0$. The conformal curvature formula yields
\[
    K_{g_s}=-\frac{e^{-2u_s}}{\alpha_s}h_s<0
\]
everywhere on $M$. Since $F_s\to F$ in $C^\infty(M)$ and the inverse of the Laplacian on zero-mean functions is continuous in all Hölder or Sobolev scales, elliptic regularity gives
\[
    u_s\to u\quad\text{in }C^\infty(M).
\]
Consequently, $g_s\to g_Z$ in the $C^\infty$-topology, proving that $g_Z$ is a smooth limit of negatively curved metrics.
\end{proof}

\begin{remark}\label{rem:boundary}
Since $K_{g_Z}$ vanishes on $Z$, the metric $g_Z$ is not negatively curved. The last part of the proof shows that it belongs to the $C^\infty$-closure of the negatively curved metrics. Thus, the metrics constructed in \cref{thm:construction} lie on the boundary of the negatively curved region in every $C^k$-topology, $k\ge2$.
\end{remark}

\section{$C^2$-robustness for a finite flat locus}\label{sec:robustness}
We now prove the main stability result. Let $g_*$ be a metric satisfying the hypotheses of \cref{thm:main}. If $\Z(g_*)=\varnothing$, then
$K_{g_*}<0$ everywhere and, by compactness, the curvature is
bounded above by a negative constant. The result then follows
from the $C^2$-openness of the set of negatively curved metrics. We therefore assume
\[
    \Z(g_*)=\{p_1,\dots,p_N\},\qquad N\ge1.
\]

The proof is divided into several steps. The key point is to choose
all constants uniformly with respect to the perturbed metric $g_*$.

\subsection{Uniform geometric and curvature constants}

Apply \cref{lem:uniform-convexity} to $g_*$ and the finite set $\Z(g_*)$. After decreasing the radius if necessary, fix $r_0>0$, $c>0$, and a $C^2$-neighborhood $\mathcal V$ of $g_*$ such that, for every $g\in\mathcal V$, the balls $B_{r_0}^g(p_i)$ are pairwise disjoint normal balls and
\begin{equation}\label{eq:hess-uniform}
\operatorname{Hess}_g\left(\frac12d_g(p_i,\cdot)^2\right)\ge c\,g
    \qquad\text{on }B_{r_0}^g(p_i), 
\end{equation}
for each $i=1,\dots,N$. By compactness and the continuity of curvature in the $C^2$-topology, after shrinking $\mathcal V$  if necessary, there exists $C\ge1$ such that
\begin{equation}\label{eq:abs-curv}
    |K_g|\le C
    \qquad\text{on }M,
    \qquad g\in\mathcal V.
\end{equation}
Choose $\eps>0$ such that
\begin{equation}\label{eq:epsilon-choice}
    0<\eps<\frac{r_0}{4}
    \qquad\text{and}\qquad
    2\eps\sqrt C<\frac{\pi}{4}.
\end{equation}
Since $4\eps<r_0$, the balls $B_{4\eps}^{g_*}(p_i)$
are also pairwise disjoint. For $g$ sufficiently close to $g_*$, the ball inclusions in \eqref{eq:ball-inclusion}, which follow from the uniform equivalence of the metrics, give
\begin{equation}\label{eq:base-ball-inclusion}
    B_{\eps/2}^{g_*}(p_i)\subset B_\eps^g(p_i)
\end{equation}
for every $i=1,\ldots, N$. Since $K_{g_*}<0$ on the compact set
\[
    M\setminus\bigcup_{i=1}^N B_{\eps/2}^{g_*}(p_i),
\]
there exists $\delta>0$ such that
\begin{equation*}
    K_{g_*}\le-\delta
    \qquad\text{on }
    M\setminus\bigcup_{i=1}^N B_{\eps/2}^{g_*}(p_i).
\end{equation*}
After shrinking $\mathcal V$ once more, continuity of curvature
with respect to the metric in the $C^2$-topology, together with
\eqref{eq:base-ball-inclusion}, yields
\begin{equation}\label{eq:negative-outside}
    K_g\le-\frac{\delta}{2}
    \qquad\text{on }
    M\setminus\bigcup_{i=1}^N B_\eps^g(p_i),
\end{equation}
for every $g\in\mathcal V$. Set
\begin{equation}\label{eq:lambda}
    \lambda^2:=\frac{\delta}{2}.
\end{equation}
For convenience, define the potentially perturbed region
\begin{equation*}
    D_g:=\bigcup_{i=1}^N B_\eps^g(p_i).
\end{equation*}
The connected components of $D_g$ are precisely the pairwise
disjoint balls $B_\eps^g(p_i)$.

Finally, since $K_{g_*}\le0$ everywhere on $M$, for every $\eta>0$ we may further shrink  $\mathcal V$ so that
\begin{equation}\label{eq:small-positive}
    K_g\le\eta
    \qquad\text{on }M,
\end{equation}
for every $g\in\mathcal V$. The value of $\eta$ will be chosen after the remaining constants have been fixed.

\subsection{Uniform passage and separation times}

Let $g$ be in the neighborhood fixed above, and let $\gamma$ be a unit-speed $g$-geodesic. Consider a connected component $[a,b]$ of $\gamma^{-1}(B_\eps^g(p_i))$ corresponding to a passage
through the ball. By \eqref{eq:abs-curv}, \eqref{eq:epsilon-choice}, and \cref{lem:gulliver},
\begin{equation}\label{eq:passage-time}
    b-a\le2\eps.
\end{equation}
This estimate is uniform in $g$, $\gamma$, the center $p_i$, and the visit.

We next establish a uniform positive lower bound on the time between consecutive passages through $D_g$.

\begin{lemma}[Uniform separation of visits]\label{lem:separation}
After shrinking the $C^2$-neighborhood $\mathcal V$ of $g_*$ if necessary, there exists $\tau>0$ such that the following holds. If $[a_k,b_k]$ and $[a_{k+1},b_{k+1}]$ are two consecutive connected components of $\gamma^{-1}(D_g)$, then
\[
    a_{k+1}-b_k\ge\tau.
\]
The constant $\tau$ is independent of $g$, $\gamma$, and $k$.
\end{lemma}

\begin{proof}
Suppose first that the two visits are to the same ball $B_\eps^g(p_i)$. Set
\[
    \rho(t):=\frac12d_g(p_i,\gamma(t))^2.
\]
At the exit time $b_k$, one has $\rho(b_k)=\eps^2/2$ and for $t>b_k$ sufficiently close to $b_k$, the geodesic lies outside the ball, so $\rho(t)\ge\rho(b_k)$ and hence $\rho'(b_k)\ge0$. As long as $\gamma(t)\in B_{r_0}^g(p_i)$, equation \eqref{eq:hess-uniform} and the geodesic equation give
\[
    \rho''(t)
    =\operatorname{Hess}_g\left(\frac12d_g(p_i,\cdot)^2\right)
      (\gamma'(t),\gamma'(t))
    \ge c>0.
\]
Thus, $\rho$ is strictly increasing after $b_k$ while the geodesic stays in $B_{r_0}^g(p_i)$. It follows that the geodesic cannot return to $B_\eps^g(p_i)$ without first reaching $\partial B_{r_0}^g(p_i)$. The outward segment from radius $\eps$ to radius $r_0$ has length at least $r_0-\eps$, as does the segment from that boundary to the next entrance into $B_\eps^g(p_i)$. Since $\gamma$ has unit speed, we obtain
\begin{equation*}
    a_{k+1}-b_k\ge2(r_0-\eps).
\end{equation*}

Suppose now that the consecutive visits are to distinct balls, say $B_\eps^g(p_i)$ and $B_\eps^g(p_j)$ with $i\ne j$. Since the set of centers is finite and the corresponding $g_*$-balls are disjoint, the minimum $g_*$-distance between distinct centers is positive. Uniform equivalence of nearby metrics and our choice of $\eps$ imply, after shrinking the neighborhood $\mathcal V$, the existence of $\sigma>0$ such that
\begin{equation*}
\operatorname{dist}_g\bigl(B_\eps^g(p_i),B_\eps^g(p_j)\bigr)\ge\sigma
    \qquad\text{for }i\ne j.
\end{equation*}
Hence $a_{k+1}-b_k\ge\sigma$. Taking
\[
    \tau:=\min\{2(r_0-\eps),\sigma\}>0
\]
proves the lemma.
\end{proof}

\subsection{Choice of a positive Riccati level}

We now fix the level which will be propagated through all visits. Define
\begin{equation*}
    m:=\min\left\{
        \frac{\lambda}{4},
        \frac{\lambda^2\tau}{4},
        \frac{1}{8\eps},
        \sqrt C
    \right\}>0.
\end{equation*}
The constants entering this definition depend only on the base metric, the chosen finite flat set, and the neighborhood already fixed.

Now choose $\eta>0$ so small that
\begin{equation}\label{eq:eta-choice}
    \eta\le\frac{m}{8\eps},
\end{equation}
and shrink the $C^2$-neighborhood one last time so that \eqref{eq:small-positive} holds for this value of $\eta$. The induction will show
\begin{equation*}
    v(a_k)\ge m
    \quad\Longrightarrow\quad
    v(b_k)\ge\frac m2,
\end{equation*}
and then
\begin{equation*}
    v(b_k)\ge\frac m2
    \quad\Longrightarrow\quad
    v(a_{k+1})\ge m.
\end{equation*}

\subsection{Loss during one passage}

\begin{lemma}[Controlled loss inside $D_g$]\label{lem:loss}
Let $[a,b]$ be a passage through one component of $D_g$. Suppose the Riccati solution is defined on $[a,b]$ and $v(a)\ge m$. Then
\[
    v(t)\ge\frac m2
    \qquad\text{for every }t\in[a,b].
\]
In particular, $v(b)\ge m/2$.
\end{lemma}

\begin{proof}
Inside $D_g$ we only use the upper bound $K_g\le\eta$. Equation \eqref{eq:riccati} gives
\begin{equation}\label{eq:riccati-inside}
    v'\ge-\eta-v^2.
\end{equation}
Let $w$ be the solution of
\begin{equation}\label{eq:w-equation}
    w'=-\eta-w^2,
    \qquad
    w(a)=m.
\end{equation}
As long as $w>0$, it is decreasing and therefore $0<w\le m$. Consequently,
\begin{equation}\label{eq:w}
 w'\ge-(\eta+m^2).   
\end{equation}
Using \eqref{eq:eta-choice} and the fact that $m\le1/8\eps$, we obtain
\[
    2\eps\eta\le\frac m4,
    \qquad
    2\eps m^2\le\frac m4.
\]
Consequently,
\begin{equation}\label{eq:loss-budget}
    2\eps(\eta+m^2)\le\frac m2.
\end{equation}
If $w$ were to vanish for the first time at some $t_0\in(a,b]$,
then \eqref{eq:passage-time}, \eqref{eq:w} and \eqref{eq:loss-budget} would give
\[
    w(t_0)
    \ge m-(\eta+m^2)(t_0-a)
    \ge m-2\eps(\eta+m^2)
    \ge\frac m2>0,
\]
a contradiction. Thus $w$ is positive on the whole passage and
\begin{equation}\label{eq:w-exit}
    w(t)\ge\frac m2,
    \qquad t\in[a,b].
\end{equation}

It remains to compare $v$ and $w$. Set $\zeta=v-w$. By \eqref{eq:riccati-inside} and \eqref{eq:w-equation},
\[
    \zeta'\ge-(v^2-w^2)=-(v+w)\zeta.
\]
Therefore
\[
    \frac{d}{dt}\left(
      \zeta(t)\exp\!\left(\int_a^t(v(s)+w(s))\,ds\right)
    \right)\ge0.
\]
Since $\zeta(a)=v(a)-w(a)\ge0$, it follows that $\zeta(t)\ge0$ throughout $[a,b]$. Hence $v\ge w$ on the whole passage, and \eqref{eq:w-exit} completes the proof.
\end{proof}

\subsection{Recovery in the negative-curvature region}

\begin{lemma}[Recovery between visits]\label{lem:recovery}
Let $[b,a^+]$ be the interval between two consecutive visits to $D_g$. If $v(b)\ge m/2$, then $v(t)>0$ throughout $[b,a^+]$, the solution reaches the level $m$ before the next visit, and
\[
    v(a^+)\ge m.
\]
\end{lemma}

\begin{proof}
On the whole interval $[b,a^+]$ the geodesic stays outside $D_g$, so by \eqref{eq:negative-outside} and \eqref{eq:lambda},
\begin{equation}\label{eq:negative-riccati}
    v'\ge\lambda^2-v^2.
\end{equation}
If $v(b)\ge m$, then $v$ cannot cross the level $m$ downward. Indeed, at a hypothetical first downward crossing one would have $v(t_0)=m$ and $v'(t_0)\le0$, whereas
\[
    v'(t_0)\ge\lambda^2-m^2
    \ge\frac{15}{16}\lambda^2>0
\]
since $m\le\lambda/4$. Assume now $m/2\le v(b)<m$. As long as $0<v<m$, we have
\[
    v'\ge\lambda^2-v^2
    \ge\frac{15}{16}\lambda^2>0.
\]
Thus $v$ increases and cannot reach zero.  Moreover,
\begin{align*} 
v(t) &= v(b)+\int_{b}^{t}v'(s)\,ds\\ &\ge \frac{m}{2} + \frac{15}{16}\lambda^2(t-b). \end{align*} 
Therefore, the time required for $v$ to grow from $m/2$ to $m$ is at most
\begin{equation*}
    T_m:=\frac{m/2}{(15/16)\lambda^2}
    =\frac{8m}{15\lambda^2}.
\end{equation*}
Since $m\le\lambda^2\tau/4$,
\[
    T_m\le\frac{2}{15}\tau<\tau.
\]
By \cref{lem:separation}, $a^+-b\ge\tau$, so the level $m$ is reached strictly before the next entrance. Once it is reached, the first part of the argument prevents a downward crossing. Hence $v(a^+)\ge m$.
\end{proof}

\subsection{The initial segment}

We now start the induction. Let $y$ be the solution of \eqref{eq:ivp}, and let $T$ denote its first possible positive zero time as defined before \eqref{eq:riccati-def}, and let $v(t)=y'(t)/y(t)$ on $(0,T)$.

Suppose first that $\gamma(0)$ belongs to one of the closed balls $\overline{B_\eps^g(p_i)}$, and if $\gamma(0)$ lies on the boundary, we include the case in which the geodesic initially points into the ball. Let $b_1$ be the first exit time from that ball. By \eqref{eq:passage-time},
\[
    b_1\le2\eps.
\]
Consider the comparison solution
\[
    z''+Cz=0,
    \qquad z(0)=0,\quad z'(0)=1,
\]
namely
\begin{equation*}
    z(t)=\frac{\sin(\sqrt C\,t)}{\sqrt C}.
\end{equation*}
By \eqref{eq:epsilon-choice}, $z(t)>0$ on $(0,2\eps]$. Since $K_g(\gamma(t))\le C$, Sturm comparison implies that $y$ has no positive zero before the first positive zero of $z$. In particular,
\[
    y(t)>0,
    \qquad 0<t\le b_1.
\]
Define
\[
    W(t):=y'(t)z(t)-y(t)z'(t).
\]
Then
\[
    W'(t)=\bigl(C-K_g(\gamma(t))\bigr)y(t)z(t)\ge0.
\]
Since $W(0)=0$, we have that $W(t)\ge0$ on $[0,b_1]$. Dividing by $y(t)z(t)>0$ yields
\begin{equation*}
    v(t)=\frac{y'(t)}{y(t)}
    \ge\frac{z'(t)}{z(t)}
    =\sqrt C\cot(\sqrt C\,t).
\end{equation*}
Since $b_1\le2\eps$ and $2\eps\sqrt C<\pi/4$,
\[
    v(b_1)
    \ge\sqrt C\cot(\sqrt C\,b_1)
    \ge\sqrt C
    \ge m.
\]
Thus the first passage is exited at level at least $m$.

Suppose next that $\gamma(0)\notin D_g$, with the boundary case interpreted so that the geodesic initially points outward. As long as the geodesic remains outside $D_g$, inequality \eqref{eq:negative-riccati} holds. From the initial conditions,
\[
    y(t)=t+O(t^3),
    \qquad
    y'(t)=1+O(t^2),
\]
and therefore
\begin{equation*}
    v(t)=\frac{y'(t)}{y(t)}\longrightarrow+\infty
    \qquad\text{as }t\to0^+.
\end{equation*}
If $a_1>0$ is the first entrance time into $D_g$, choose $t_0\in(0,a_1)$ sufficiently such that $v(t_0)\ge m$. The same first-crossing argument used in \cref{lem:recovery} shows that $v$ cannot fall below $m$ before $a_1$, so
\[
    v(a_1)\ge m.
\]
If the geodesic never enters $D_g$, then $v(t)\ge m$ for every sufficiently large $t$ after $t_0$, and in fact remains positive for all $t>0$.

\subsection{Inductive conclusion}

We can now complete the proof of the main theorem.

\begin{proof}[Proof of \cref{thm:main}]
Take a metric $g$ in the final $C^2$-neighborhood constructed above, a unit-speed $g$-geodesic $\gamma$, and a non-trivial normal Jacobi field $J$ with $J(0)=0$. After multiplication by a non-zero scalar, we reduce to \eqref{eq:ivp}. Let $v$ be the Riccati solution of \eqref{eq:riccati} on its maximal interval $(0,T)$.

The initial-segment argument shows that the first relevant entrance or exit level is at least $m$. For every subsequent passage $[a_k,b_k]$ through $D_g$, \cref{lem:loss} gives
\[
    v(a_k)\ge m
    \quad\Longrightarrow\quad
    v(t)\ge\frac m2
    \quad(a_k\le t\le b_k).
\]
Between visits, \cref{lem:recovery} gives
\[
    v(b_k)\ge\frac m2
    \quad\Longrightarrow\quad
    v(a_{k+1})\ge m,
\]
and keeps $v$ positive throughout the intervening interval. Therefore the induction propagates over every visit of the geodesic to $D_g$.

The uniform separation time $\tau>0$ rules out an accumulation of infinitely many visits in finite time. Hence the preceding estimates apply on every compact time interval contained in $(0,T)$. They show that $v$ stays positive and, away from the initial instant, is bounded below by a positive constant determined by the current stage of the induction. In particular, if $T<+\infty$, then $y$ could not converge to zero as $t\to T^-$, because $y'>0$ whenever $v>0$ and $y>0$. This contradicts the definition of $T$. Thus $T=+\infty$.

Consequently $v(t)>0$ for every $t>0$. By \eqref{eq:norm-derivative},
\[
    \frac{d}{dt}\|J(t)\|^2
    =2v(t)y(t)^2>0
    \qquad\text{for every }t>0.
\]
Since $\gamma$ and $J$ were arbitrary, $(M,g)$ has no focal points. All constants were chosen uniformly on the final $C^2$-neighborhood, so every metric in that neighborhood has no focal points. This proves that $g_*$ is a $C^2$-interior point of $\NoFoc$.
\end{proof}

\begin{remark}[Where finiteness is used]\label{rem:finiteness}
The proof uses finiteness of $\Z(g_*)$ in two places. First, it yields a uniform strictly negative curvature bound outside finitely many small balls. Second, it yields a uniform positive separation between different balls. The same Riccati mechanism therefore extends verbatim to any situation in which the possible positive-curvature region is a finite union of uniformly short ``bad'' components separated by a uniformly positive amount of negative-curvature travel time.
\end{remark}

\section{Consequences and geometric interpretation}\label{sec:consequences}

We conclude by proving the corollaries stated in the introduction and explaining the role of isolated flat points.

\subsection{Arbitrarily small mixed-curvature perturbations}

\begin{proof}[Proof of \cref{cor:mixed}]
Choose $p\in\Z(g_*)$. Let $r(x)=d_{g_*}(p,x)$ in a small normal ball and choose a smooth cutoff $\chi$ supported in that ball with $\chi\equiv1$ near $p$. Define
\[
    \psi(x):=-\chi(x)r(x)^2\le 0.
\]
Then $\psi$ is smooth and, with our Laplacian convention,
\[
    \Delta_{g_*}\psi(p)<0.
\]
For $s>0$ let
\[
    g_s:=e^{2s\psi}g_*.
\]
The conformal curvature formula gives
\[
    K_{g_s}
    =e^{-2s\psi}\bigl(K_{g_*}-s\Delta_{g_*}\psi\bigr).
\]
Since $K_{g_*}(p)=0$ and $\Delta_{g_*}\psi(p)<0$,
\[
    K_{g_s}(p)>0
\]
for every sufficiently small $s>0$. On the other hand, choose a point $q$ outside the support of $\psi$; then $K_{g_s}(q)=K_{g_*}(q)<0$. Hence $K_{g_s}$ takes both signs. Clearly $g_s\to g_*$ in $C^\infty$ as $s\to0$, so every $C^2$-neighborhood of $g_*$ contains such a mixed-curvature metric. Taking $s$ small enough that $g_s\in\mathcal U$, \cref{thm:main} implies that this metric has no focal points.
\end{proof}

This clarifies the sense in which \cref{thm:main} goes beyond
the pointwise curvature condition: the neighborhood of metrics
without focal points necessarily contains metrics with positive
curvature somewhere.

\subsection{Anosov dynamics}

\begin{proof}[Proof of \cref{cor:anosov}]
Let $g\in\mathcal U$. Since $\mathcal U$ is open, there exists a $C^2$-neighborhood $\mathcal V_g$ of $g$ such that
\[
    \mathcal V_g\subset\mathcal U\subset\NoFoc\subset\NoConj.
\]
Thus $g$ is a $C^2$-interior point of the set of metrics without conjugate points. Ruggiero's theorem \cite{Ruggiero1991} implies that the geodesic flow of $g$ is Anosov.
\end{proof}

Combining the two main theorems gives the following concrete formulation.

\begin{corollary}\label{cor:hyperbolic-prescribed}
Let $(M,g_0)$ be a compact hyperbolic surface and let $Z\subset M$ be a non-empty finite set. There exists a smooth conformal metric $g_Z$ such that:
\begin{enumerate}[label=(\roman*)]
    \item $K_{g_Z}\le0$ and $K_{g_Z}^{-1}(0)=Z$;
    \item $g_Z$ is a $C^\infty$-limit of negatively curved metrics;
    \item $g_Z$ is a $C^2$-interior point of the set of metrics without focal points;
    \item a $C^2$-neighborhood of $g_Z$ consists entirely of Anosov metrics without focal points and contains metrics whose Gaussian curvature takes both signs.
\end{enumerate}
\end{corollary}

\begin{proof}
Items (i) and (ii) are \cref{thm:construction}. Item (iii) follows from \cref{thm:main}, and item (iv) follows from \cref{cor:mixed,cor:anosov}.
\end{proof}

\subsection{Why a flat closed geodesic is different}

The assumption that the flat set in \cref{thm:main} consists
of isolated points is not merely a technical convenience.
Suppose that $g$ has non-positive curvature and admits a closed
geodesic $\gamma$ along which

\[
    K_g(\gamma(t))=0
    \qquad\text{for every }t.
\]
Along $\gamma$, the normal Jacobi equation reduces to
\[
    y''=0.
\]
The solution with $y(0)=1$ and $y'(0)=0$ is constant. It defines a non-zero transverse vector whose norm is unchanged by the linearized geodesic flow along the periodic orbit. This neutral transverse behavior is incompatible with the uniform exponential contraction/expansion required by the Anosov property. Hence the geodesic flow is not Anosov.

If such a metric were $C^2$-robustly without focal points, it would in particular be a $C^2$-interior point of the set of metrics without conjugate points, and Ruggiero's theorem would force its geodesic flow to be Anosov, a contradiction. Therefore:

\begin{proposition}\label{prop:closed-flat}
A non-positively curved metric on a compact surface containing a closed geodesic along which the Gaussian curvature vanishes identically cannot be a $C^2$-interior point of the set of metrics without focal points.
\end{proposition}

This explains why replacing the isolated zero-curvature points in \cref{thm:construction} by an entire flat closed geodesic changes the stability problem fundamentally.




\bigskip
\noindent School of Mathematics, Sun Yat-sen University, Zhuhai, China\\
\noindent Email: \texttt{vilchez@mail.sysu.edu.cn}

\medskip
\noindent School of Mathematics, Sun Yat-sen University, Zhuhai, China\\
\noindent Email: \texttt{sergio@mail.sysu.edu.cn}


\begin{thebibliography}{99}

\bibitem{Anosov1967}
D.~V. Anosov,
\emph{Geodesic flows on closed Riemannian manifolds of negative curvature},
Proceedings of the Steklov Institute of Mathematics, Vol.~90,
American Mathematical Society, Providence, RI, 1967.

\bibitem{Aubin1998}
T.~Aubin,
\emph{Some Nonlinear Problems in Riemannian Geometry},
Springer Monographs in Mathematics, Vol.~27,
Springer-Verlag, Berlin, 1998.

\bibitem{Eberlein1973}
P.~Eberlein,
\emph{When is a geodesic flow of Anosov type? I},
J. Differential Geom. \textbf{8} (1973), no.~3, 437--463.

\bibitem{Gulliver1975}
R.~Gulliver,
\emph{On the variety of manifolds without conjugate points},
Trans. Amer. Math. Soc. \textbf{210} (1975), 185--201.

\bibitem{Klingenberg1974}
W.~Klingenberg,
\emph{Riemannian manifolds with geodesic flows of Anosov type},
Ann. of Math. (2) \textbf{99} (1974), no.~1, 1--13.


\bibitem{Ruggiero1991}
R.~O. Ruggiero,
\emph{On the creation of conjugate points},
Math. Z. \textbf{208} (1991), no.~1, 41--56.

\end{thebibliography}
\end{document}